\documentclass[11pt,a4paper]{amsart}

\usepackage[english]{babel}
\usepackage[T1]{fontenc}
\usepackage[utf8]{inputenc}
\usepackage{amsmath}
\usepackage{amssymb}
\usepackage{amsthm}
\usepackage{latexsym}
\usepackage{amsfonts}

\usepackage[onehalfspacing]{setspace} 
\setdisplayskipstretch{.421} 
\newlength{\abovebis} 
\newlength{\belowbis} 
\newlength{\aboveshortbis} 
\newlength{\belowshortbis} 
\everydisplay\expandafter{% 
  \the\everydisplay 
  \advance\abovedisplayskip\abovebis 
  \advance\belowdisplayskip\belowbis 
  \advance\abovedisplayshortskip\aboveshortbis 
  \advance\belowdisplayshortskip\belowshortbis 
}

\allowdisplaybreaks[3]

\def\R{\mathbb{R}}

\def\N{\mathbb{N}}
\def\C{\mathbb{C}}
\def\matn{M_{n}(\C)}
\def\matrn{M_{n}(\R)}

\def\Ree{\mathrm{Re}}
\def\Imm{\mathrm{Im}}

\theoremstyle{plain}

\newtheorem{lem}{Lemma}[section]
\newtheorem{theo}[lem]{Theorem}
\newtheorem{prop}[lem]{Proposition}

\theoremstyle{definition}

\numberwithin{equation}{section}

\begin{document}
\title[Global multi-channel Stability in 2D]{Global stability for the multi-channel Gel'fand-Calder\'on inverse problem in two dimensions}
\author{Matteo Santacesaria}
\address[M. Santacesaria]{Centre de Mathématiques Appliquées, \'Ecole Polytechnique, 91128, Palaiseau, France}
\email{santacesaria@cmap.polytechnique.fr}
%\subjclass{Primary ; Secondary }
%\keywords{}
\begin{abstract}
We prove a global logarithmic stability estimate for the multi-channel Gel'fand-Calder\'on inverse problem on a two-dimensional bounded domain, i.e. the inverse boundary value problem for the equation $-\Delta \psi + v\, \psi = 0$ on $D$, where $v$ is a smooth matrix-valued potential defined on a bounded planar domain $D$.
\end{abstract}

\maketitle

\section{Introduction}
The Schr\"odinger equation at zero energy
\begin{equation} \label{equa}
-\Delta \psi + v(x) \psi =0 \; \; \textrm{on } D \subset \R^2
\end{equation}
arises in quantum mechanics, acoustics and electrodynamics. The reconstruction of the complex-valued potential $v$ in equation \eqref{equa} through the Dirichlet-to-Neumann operator is one of the most studied inverse problems (see \cite{N1}, \cite{M}, \cite{B}, \cite{N2}, \cite{NS}, \cite{NS2} and references therein).

In this article we consider the multi-channel two-dimensional Schr\"odinger equation, i.e. equation \eqref{equa} with matrix-valued potentials and solutions; this case was already studied in \cite{X, NS2}. One of the motivations for studying the multi-channel equation is that it comes up as a 2D-approximation for the 3D equation (see \cite[Sec. 2]{NS2}).

This paper is devoted to give a global stability estimate for this inverse problem in the multi-channel case, which is highly related to the reconstruction method of \cite{NS2}.\smallskip

Let $D$ be an open bounded domain in $\R^2$ with $C^2$ boundary and $v \in C^1(\bar D, M_{n}(\C))$, where $\matn$ is the set of the $n \times n$ complex-valued matrices. The Dirichlet-to-Neumann map associated to $v$ is the operator $\Phi : C^1(\partial D, \matn) \to L^p(\partial D, \matn), \; p < \infty$ defined by:
\begin{equation}
\Phi(f) = \left.\frac{\partial \psi}{\partial \nu}\right|_{\partial D}
\end{equation}
where $f \in C^1(\partial D, \matn)$, $\nu$ is the outer normal of $\partial D$ and $\psi$ is the $H^1 (\bar D, \matn)$-solution of the Dirichlet problem
\begin{equation} \label{equa1}
-\Delta \psi + v(x) \psi =0 \; \; \textrm{on } D, \; \; \psi|_{\partial D} = f;
\end{equation}
here we assume that
\begin{equation} \label{direig}
0 \textrm{ is not a Dirichlet eigenvalue for the operator } - \Delta + v \textrm{ in } D.
\end{equation}
%Equation \eqref{equa} arises, in particular, in quantum mechanics, acoustics, electrodynamics; formally, it looks like the Schr\"odinger equation with potential $v$ at zero energy.
The following inverse boundary value problem arises from this construction: given $\Phi$, find $v$.

This problem can be considered as the Gel’fand inverse boundary value problem for the multi-channel Schr\"odinger equation at zero energy (see \cite{G}, \cite{N1}) and can also be seen as a generalization of the Calder\'on problem for the electrical impedance tomography (see \cite{C}, \cite{N1}). Note also that we can think of this problem as a model for the monochromatic ocean tomography (e.g. see \cite{Bur} for similar problems arising in this tomography).\smallskip

In the case of complex-valued potentials the global injectivity of the map $v \to \Phi$ was firstly proved in \cite{N1} for $D \subset \R^d$ with $d \geq 3$ and in \cite{B} for $d = 2$ with $v \in L^p$: in particular, these results were obtained by the use of global reconstructions developed in the same papers. The first global uniqueness result (along with an exact reconstruction method) for matrix-valued potentials was given in \cite{NS2}, which deals with $C^1$ matrix-valued potentials defined on a domain in $\R^2$. A global stability estimate for the Gel'fand-Calder\'on problem for $d \geq 3$ was found for the first time by Alessandrini in \cite{A}; this result was recently improved in \cite{N2}. In the two-dimensional case the first global stability estimate was given in \cite{NS}.

In this paper we extend the results of \cite{NS} to the matrix-valued case; we do not discuss global results for special real-valued potentials arising from conductivities: for this case the reader is referred to the references given in \cite{A}, \cite{B}, \cite{M}, \cite{N1}, \cite{N2}, \cite{NS}.

Our main result is the following:
\begin{theo} \label{theo1}
Let $D \subset \R^2$ be an open bounded domain with $C^2$ boundary, let $v_1 , v_2 \in C^2(\bar D,\matn)$ be two matrix-valued potentials which satisfy \eqref{direig}, with $\|v_j \|_{C^2(\bar D)} \le N$ for $j =1,2$, and $\Phi_1 , \Phi_2$ the corresponding Dirichlet-to-Neumann operators. For simplicity we assume also that $v_j|_{\partial D} = 0$ and $\frac{\partial}{\partial \nu} v_j|_{\partial D}=0$ for $j=1,2$. Then there exists a constant $C = C(D, N, n)$ such that
\begin{equation} \label{est}
\|v_2 - v_1\|_{L^{\infty}(D)} \leq C \left( \log(3 + \|\Phi_2 - \Phi_1\|^{-1} ) \right)^{-\frac 3 4} \left( \log (3 \log(3+\| \Phi_2 -\Phi_1\|^{-1}))\right)^2,
\end{equation}
where $\| A\|$ denotes the norm of an operator $A : L^{\infty}(\partial D, \matn) \to$ $L^{\infty}(\partial D, \matn)$ and $\|v\|_{L^{\infty}(D)} = \max_{1 \leq i,j \leq n} \| v_{i,j}\|_{L^{\infty}(D)}$ (likewise for $\|v \|_{C^2(\bar D)}$) for a ma\-trix-valued potential $v$.
\end{theo}

This is the first global stability result for the multi-channel ($n \geq 2$) Gel'fand-Calder\'on inverse problem in two dimension. In addition, Theorem \ref{theo1} is new also for the scalar case, as the estimate obtained in \cite{NS} is weaker.

Instability estimates complementing the stability estimates of \cite{A}, \cite{N2}, \cite{NS} and of the present work are given in \cite{M}, \cite{I1}.\smallskip

The proof of Theorem \ref{theo1} is based on results obtained in \cite{NS}, \cite{NS2}, which takes inspiration mostly from \cite{B} and \cite{A}. In particular, for $z_0 \in D$ we use the existence and uniqueness of a family of solution $\psi_{z_0}(z,\lambda)$ of equation \eqref{equa} where in particular $\psi_{z_0} \to e^{\lambda(z - z_0)^2}I$, for $\lambda \to \infty$ (where $I$ is the identity matrix). Then, using an appropriate matrix-valued version of Alessandrini's identity along with stationary phase techniques, we obtain the result. Note that this matrix-valued identity is one of the new results of this paper.
\smallskip

A generalization of Theorem \ref{theo1} in the case where we do not assume that $v_j|_{\partial D} = 0$ and $\frac{\partial}{\partial \nu} v_j|_{\partial D}=0$ for $j=1,2$, is given in section \ref{secext}.\bigskip

This work was fulfilled in the framework of researches under the direction of R. G. Novikov.

\section{Preliminaries} \label{secprel}

In this section we introduce and give details about the above-mentioned family of solutions of equation \eqref{equa}, which will be used throughout all the paper.
\smallskip

We identify $\R^2$ with $\C$ and use the coordinates $z= x_1 + i x_2, \; \bar z = x_1 - i x_2$ where $(x_1, x_2) \in \R^2$. % In addition, the notation $f=f(z)$ does not necessarily mean that $f$ is holomorphic with respect to $z$.
Let us define the function spaces $C^1_{\bar z}(\bar D) = \{ u : u, \frac{\partial u}{ \partial \bar z} \in C(\bar D, \matn) \}$ with the norm $\|u \|_{C^1_{\bar z}(\bar D)} = \max ( \|u\|_{C(\bar D)}, \| \frac{\partial u}{\partial \bar z} \|_{C(\bar D)} )$, where $\|u\|_{C(\bar D)} = \sup_{z \in \bar D} |u|$ and $|u| = \max_{1 \leq i,j \leq n} |u_{i,j}|$; we define also $C^1_{z}(\bar D) = \{ u : u, \frac{\partial u}{ \partial z} \in C(\bar D, \matn) \}$ with an analogous norm. Following \cite{NS}, \cite{NS2}, we consider the functions:
\begin{align} \label{eq11}
G_{z_0}(z,\zeta, \lambda) &= e^{\lambda (z-z_0)^2}g_{z_0} (z,\zeta,\lambda)e^{-\lambda (\zeta-z_0)^2},\\ \label{eq12}
g_{z_0}(z,\zeta, \lambda) &= \frac{e^{\lambda(\zeta-z_0)^2-\bar \lambda(\bar \zeta - \bar z_0)^2}}{4 \pi^2} \int_D \frac{e^{-\lambda(\eta -z_0)^2+\bar \lambda(\bar \eta -\bar z_0)^2}}{(z- \eta)(\bar \eta -\bar \zeta)} d\Ree \eta \, d \Imm \eta, \\ \label{eq13}
\psi_{z_0} (z,\lambda) &= e^{\lambda (z-z_0)^2} \mu_{z_0}(z,\lambda), \\ \label{eq14}
\mu_{z_0} (z,\lambda) &= I + \int_D g_{z_0} (z,\zeta, \lambda) v(\zeta) \mu_{z_0} (\zeta, \lambda) d \mathrm{Re}\zeta \, d \mathrm{Im} \zeta ,\\ \label{eq15}
h_{z_0} (\lambda) &= \int_D e^{\lambda (z-z_0)^2 - \bar \lambda (\bar z - \bar z_0)^2} v(z) \mu_{z_0} (z,\lambda) d \mathrm{Re}z \, d \mathrm{Im} z,
\end{align}
where $z , z_0, \zeta \in D$ and $\lambda \in \C$ and $I$ is the identity matrix. In addition, equation \eqref{eq14} at fixed $z_0$ and $\lambda$, is considered as a linear integral equation for $\mu_{z_0}( \cdot, \lambda) \in C^1_{\bar z} (\bar D)$.
The functions $G_{z_0}(z,\zeta, \lambda), \; g_{z_0}(z,\zeta, \lambda), \; \psi_{z_0}(z,\lambda), \; \mu_{z_0}(z,\lambda)$ defined above, satisfy the following equations (see \cite{NS}, \cite{NS2}):

\begin{align} \label{eq21}
4 \frac{\partial^2}{\partial z \partial \bar z} &G_{z_0} (z, \zeta, \lambda) = \delta(z-\zeta), \\
4\frac{\partial^2}{\partial \zeta \partial \bar \zeta} &G_{z_0}(z,\zeta,\lambda) = \delta(\zeta-z),\\ \label{eq22}
4\left(\frac{\partial}{\partial z} + 2\lambda (z-z_0) \right) \frac{\partial}{\partial \bar z} &g_{z_0}(z,\zeta, \lambda) = \delta(z-\zeta), \\
4\frac{\partial}{\partial \bar \zeta} \left(\frac{\partial}{\partial \zeta} - 2\lambda (\zeta-z_0) \right) &g_{z_0}(z,\zeta, \lambda) = \delta(\zeta-z),\\ \label{eq23}
-4 \frac{\partial^2}{\partial z \partial \bar z} &\psi_{z_0} (z,\lambda) + v(z) \psi_{z_0}(z,\lambda) = 0,\\ \label{eq24}
-4 \left(\frac{\partial}{\partial z} + 2\lambda (z-z_0) \right) \frac{\partial}{\partial \bar z} &\mu_{z_0}(z, \lambda) + v(z) \mu_{z_0} (z,\lambda) = 0,
\end{align}
where $z, z_0 , \zeta \in D$, $\lambda \in \C$, $\delta$ is the Dirac's delta. (In addition, it is assumed that \eqref{eq14} is uniquely solvable for $\mu_{z_0}( \cdot, \lambda) \in C^1_{\bar z} (\bar D)$ at fixed $z_0$ and $\lambda$.) 

We say that the functions $G_{z_0}$, $g_{z_0}$, $\psi_{z_0}$, $\mu_{z_0}$, $h_{z_0}$ are the Bukhgeim-type analogues of the Faddeev functions (see \cite{NS2}).
\smallskip

Now we state some fundamental lemmata.
%whose proofs in the scalar case can be found in \cite{NS}: the generalisation to the matrix-valued case is straightforward. 
Let
\begin{equation} \label{green}
g_{z_0, \lambda} u(z) = \int_D g_{z_0} (z, \zeta, \lambda) u(\zeta) d \mathrm{Re}\zeta \, d \mathrm{Im}\zeta, \; z \in \bar D, \; z_0, \lambda \in \C,
\end{equation}
where $g_{z_0}(z, \zeta, \lambda)$ is defined by \eqref{eq12} and $u$ is a test function.

\begin{lem}[\cite{NS}] \label{lem1}
Let $g_{z_0, \lambda} u$ be defined by \eqref{green}. Then, for $z_0, \lambda \in \C$, the following estimates hold:
\begin{align} \label{contg}
&g_{z_0, \lambda} u \in C^1_{\bar z}(\bar D), \quad \textrm{for} \; u \in C(\bar D), \\ \label{estcontg1}
\| &g_{z_0, \lambda} u \|_{C^1(\bar D)} \leq c_1(D,\lambda) \|u\|_{C(\bar D)}, \quad \textrm{for} \; u \in C(\bar D), \\ \label{est1}
\| &g_{z_0, \lambda} u \|_{C^1_{\bar z}(\bar D)} \leq \frac{c_2(D)}{|\lambda|^{\frac 1 2}} \|u \|_{C^1_{\bar z}(\bar D)}, \quad \textrm{for} \; u \in C^1_{\bar z}(\bar D), \; |\lambda| \geq 1. 
%\\ \label{est2}
%\|&\frac{\partial}{\partial z}g_{z_0,\lambda}u\|_{L^p(\bar D)} \leq \frac{c_2(D,p)}{|\lambda|^{ \frac 1 2}}\|u\|_{C^1_{\bar z} (\bar D)} , \; |\lambda| \geq 1, \; 1 < p < \infty.
\end{align}
\end{lem}

Given a potential $v \in C^1_{\bar z}(\bar D)$ we define the operator $g_{z_0,\lambda} v$ simply as $(g_{z_0,\lambda} v) u(z) = g_{z_0,\lambda} w(z), \; w=vu$, for a test function $u$. If $u \in C^1_{\bar z}(\bar D)$, by Lemma \ref{lem1} we have that $g_{z_0,\lambda} v : C^1_{\bar z}(\bar D)  \to C^1_{\bar z}(\bar D)$,
\begin{equation} \label{estsolmu}
\| g_{z_0,\lambda} v \|^{op}_{C^1_{\bar z}(\bar D)} \leq 2n \| g_{z_0,\lambda} \|^{op}_{C^1_{\bar z}(\bar D)} \|v\|_{C^1_{\bar z}(\bar D)},
\end{equation}
where $\| \cdot \|^{op}_{C^1_{\bar z}(\bar D)}$ denotes the operator norm in $C^1_{\bar z}(\bar D)$, $z_0, \lambda \in \C$. In addition, $\| g_{z_0,\lambda} \|^{op}_{C^1_{\bar z}(\bar D)}$ is estimated in Lemma \ref{lem1}. Inequality \eqref{estsolmu} and Lemma \ref{lem1} imply existence and uniqueness of $\mu_{z_0}(z, \lambda)$ (and thus also $\psi_{z_0}(z,\lambda)$) for $|\lambda| > \rho(D,K,n)$, where $\|v\|_{C^1_{\bar z}(\bar D)} < K$.
\smallskip

Let
\begin{align*}
\mu^{(k)}_{z_0}(z, \lambda) &= \sum_{j=0}^k (g_{z_0,\lambda} v)^j I, \\
h^{(k)}_{z_0}(\lambda) &= \int_D e^{\lambda (z-z_0)^2 -\bar \lambda (\bar z- \bar z_0)^2}v(z) \mu^{(k)}_{z_0}(z,\lambda) d\Ree z \, d \Imm z,
\end{align*} 
where $z,z_0 \in D$, $\lambda \in \C$, $k \in \N \cup \{ 0 \}$.

\begin{lem}[\cite{NS}] \label{lem2}
For $v \in C^1_{\bar z}(\bar D)$ such that $v|_{\partial D} = 0$ the following formula holds:
\begin{equation} \label{estp}
v(z_0) = \frac{2}{\pi} \lim_{\lambda \to \infty} |\lambda| h^{(0)}_{z_0}(\lambda), \; z_0 \in D.
\end{equation}
In addition, if $v \in C^2(\bar D)$, $v|_{\partial D}= 0$ and $\frac{\partial v}{\partial \nu}|_{\partial D} = 0$ then
\begin{equation} \label{estm}
\left|v(z_0) - \frac{2}{\pi}|\lambda| h^{(0)}_{z_0}(\lambda)\right| \leq c_3(D,n) \frac{\log(3|\lambda|)}{|\lambda|}\|v\|_{C^2(\bar D)},
\end{equation}
for $z_0 \in D$, $\lambda \in \C$, $|\lambda| \geq 1$.
\end{lem}

Let
$$W_{z_0}(\lambda)=\int_D e^{\lambda(z-z_0)^2-\bar\lambda(\bar z-\bar z_0)^2}
w(z)d\Ree\,z d\Imm\,z,$$
where $z_0\in\bar D$, $\lambda\in\C$ and $w$ is some $\matn$-valued function on $\bar D$.
(One can see that $W_{z_0}=h_{z_0}^{(0)}$ for $w=v$.)

\begin{lem}[\cite{NS}] \label{lem3}
For $w\in C_{\bar z}^1(\bar D)$ the following estimate holds:
%\begin{subequations} \label{estw}
\begin{align} \label{estw1}
|W_{z_0}(\lambda)| &\le c_4(D) \frac{\log\,(3|\lambda|)}{ |\lambda|}  \|w\|_{C_{\bar z}^1(\bar D)},\ z_0\in\bar D,\ |\lambda|\ge 1. 
%\\ \label{estw2}
%|W_{z_0}(\lambda)| &\le c_4(D) \frac{\log\,(3|\lambda|)}{ |\lambda|}  \|w\|_{C(\bar D)}+\frac{c_5(D,p)}{|\lambda|}\| \frac{\partial}{\partial z} w\|_{L^p(\bar D)},
\end{align}
%for  $2 < p < \infty$.
%\end{subequations}
\end{lem}

\begin{lem}[\cite{NS2}] \label{lem4}
For $v \in C^1_{\bar z}(\bar D)$ and for $\|g_{z_0,\lambda} v \|^{op}_{C^1_{\bar z}(\bar D)} \leq \delta < 1$ we have that
\begin{align} \label{estmuk}
&\|\mu_{z_0}(\cdot, \lambda) - \mu_{z_0}^{(k)}(\cdot, \lambda)\|_{C^1_{\bar z}(\bar D)} \leq \frac{\delta^{k+1}}{1-\delta}, \\ \label{esth}
&|h_{z_0}(\lambda)-h^{(k)}_{z_0}(\lambda)| \leq c_5(D,n)\frac{\log(3|\lambda|)}{|\lambda|} \frac{\delta^{k+1}}{1-\delta} \|v\|_{C^1_{\bar z}(\bar D)},
\end{align}
where $z_0 \in D , \; \lambda \in \C,\; |\lambda |\ge 1, \; k \in \N \cup \{ 0 \}$.
\end{lem}

The proofs of Lemmata \ref{lem1}-\ref{lem4} can be found in the references given.

We will also need the following two new lemmata.
\begin{lem} \label{lem5}
Let $g_{z_0, \lambda} u$ be defined by \eqref{green}, where $u \in C^1_{\bar z}(\bar D)$, $z_0, \lambda \in \C$. Then the following estimate holds:
\begin{align} \label{estcontg}
\| &g_{z_0, \lambda} u \|_{C(\bar D)} \leq c_6(D) \frac{\log(3|\lambda|)}{|\lambda|} \|u \|_{C^1_{\bar z}(\bar D)}, \; |\lambda| \geq 1.
\end{align}
\end{lem}

\begin{lem} \label{lem6}
The expression
\begin{equation} \label{defdoppiavu}
W(u,v)(\lambda) = \int_D e^{\lambda (z-z_0)^2 -\bar \lambda (\bar z- \bar z_0)^2}u(z) (g_{z_0,\lambda}v)(z) d\Ree z \, d \Imm z,
\end{equation}
defined for $u,v \in C^1_{\bar z}(\bar D)$ with $\|u\|_{C^1_{\bar z}(\bar D)},\|v\|_{C^1_{\bar z}(\bar D)} \leq N_1$, $\lambda \in \C$, $z_0 \in D$, satisfies the estimate
\begin{align} \label{estdoppiavu}
|W(u,v)(\lambda)|\leq c_7(D,N_1,n)\frac{\left( \log(3|\lambda|) \right)^2}{|\lambda|^{1+3/4}}, \qquad |\lambda | \geq 1.
\end{align} 
\end{lem}

The proofs of Lemmata \ref{lem5}, \ref{lem6} are given in section \ref{seclem}.

\section{Proof of Theorem \ref{theo1}}
We begin with a technical lemma, which will be useful to generalise Alessandrini's identity.
%matrix-valued version of Alessandrini's identity (see \cite{A} for the scalar case).
\begin{lem} \label{lemtech}
Let $v \in C^1(\bar D, \matn)$ be a matrix-valued potential which satisfies condition \eqref{direig} (i.e. $0$ is not a Dirichlet eigeinvalue for the operator $- \Delta +v$ in $D$). Then ${}^t v$, the transpose of $v$, also satisfies condition \eqref{direig}.
\end{lem}
The proof of Lemma \ref{lemtech} is given in section \ref{seclem}.

We can now state and prove a matrix-valued version of Alessandrini's identity (see \cite{A} for the scalar case).
\begin{lem}
Let $v_1,v_2 \in C^1(\bar D, \matn)$ be two matrix-valued potentials which satisfy \eqref{direig}, $\Phi_1, \Phi_2$ their associated Dirichlet-to-Neumann operators, respectively, and $u_1,u_2 \in C^2(\bar D, \matn)$ matrix-valued functions such that 
\begin{align} \nonumber
(-\Delta + v_1) u_1 =0, \quad (-\Delta + {}^t v_2) u_2 =0 \quad \text{on } D,
\end{align}
where ${}^t A$ stand for the transpose of $A$. Then we have the identity
\begin{equation} \label{aless3}
\int_{\partial D} {}^t u_2(z) (\Phi_2 - \Phi_1)u_1(z) |dz| = \int_D {}^t u_2(z) (v_2(z) -v_1(z)) u_1(z) d \Ree z \, d \Imm z.
\end{equation}
\end{lem}
\begin{proof}
If $v \in C^1(\bar D, \matn)$ is any matrix-valued potential (which satisfies \eqref{direig}) and $f_1, f_2 \in C^1(\partial D,\matn)$ then we have
\begin{equation} \label{symdtn1}
\int_{\partial D} {}^t f_2 \Phi f_1 |dz| = \int_{\partial D} {}^t \! \left( {}^t f_1 \Phi^{\ast} f_2 \right) |dz|,
\end{equation}
where $\Phi$ and $\Phi^{\ast}$ are the Dirichlet-to-Neumann operators associated to $v$ and ${}^t v$, respectively (these operators are well-defined thanks to Lemma \ref{lemtech}). Indeed, it is sufficient to extend $f_1$ and $f_2$ in $D$ as the solutions of the Dirichlet problems $(-\Delta + v)\tilde f_1=0$, $(-\Delta +{}^t  v)\tilde f_2=0$  on $D$ and $\tilde f_j |_{\partial D} = f_j$, for $j=1,2$, so that one obtains
\begin{align*}
&\int_{\partial D}\left( {}^t f_2 \Phi f_1 - {}^t \! \left( {}^t f_1 \Phi^{\ast} f_2 \right) \right) |dz|
\\
&\qquad =\int_{\partial D}\left( {}^t f_2 \frac{\partial \tilde f_1}{\partial \nu} - {}^t \! \left( \frac{\partial \tilde f_2}{\partial \nu} \right) f_1 \right)|dz|\\
&\qquad =\int_D \left( {}^t \tilde f_2\, \Delta \tilde f_1 - {}^t \! \left( \Delta \tilde f_2\right) \tilde f_1  \right) d\Ree z \, d \Imm z \\
&\qquad =\int_D \left( {}^t \tilde f_2\, v\,\tilde f_1 - {}^t \! \left( {}^t v \, \tilde f_2\right) \tilde f_1 \right) d\Ree z \, d \Imm z=0,
\end{align*}
where for the second equality we used the following matrix-valued version of the classical scalar Green's formula:
\begin{equation} \label{greenform}
\int_{\partial D} \left( {}^t \! \left( \frac{\partial f}{\partial \nu}\right) g - {}^t f \frac{\partial g}{\partial \nu}  \right) |dz| = \int_D \left( {}^t \! \left( \Delta f\right) g - {}^t f \Delta g \right) d \Ree z \, d\Imm z,
\end{equation}
for any $f,g \in C^2(D, \matn) \cap C^1(\bar D, \matn)$.

Identities \eqref{symdtn1} and \eqref{greenform} imply
\begin{align*}
&\int_{\partial D} {}^t u_2(z) (\Phi_2 - \Phi_1)u_1(z) |dz| \\
&\qquad = \int_{\partial D}\left({}^t \! \left( {}^t u_1(z)\Phi_2^{\ast} u_2(z) \right) - {}^t \! u_2(z)\Phi_1 u_1(z)  \right) |dz|\\
&\qquad = \int_{\partial D}\left({}^t \! \left( \frac{\partial u_2(z)}{\partial \nu} \right) u_1(z) - {}^t \! u_2(z) \frac{\partial u_1(z)}{\partial \nu}  \right) |dz|\\
&\qquad = \int_D \left( {}^t \! \left( \Delta u_2(z) \right)u_1(z) - {}^t u_2(z) \Delta u_1(z)  \right) d\Ree z \, d \Imm z \\
&\qquad = \int_D \left( {}^t \! \left({}^t v_2(z)\, u_2(z)\right)u_1(z) - {}^t u_2(z)\, v_1(z)\, u_1(z)  \right) d\Ree z \, d \Imm z \\
&\qquad = \int_D {}^t u_2(z) (v_2(z) -v_1(z)) u_1(z) d \Ree z \, d \Imm z. \qedhere
\end{align*}
\end{proof}

Now let $\bar\mu_{z_0}$ denote the complex conjugated of $\mu_{z_0}$ (the solution of \eqref{eq14}) for a $\matrn$-valued potential $v$ and, more generally, the solution of \eqref{eq14} with $g_{z_0}(z,\zeta,\lambda)$ replaced by $\overline{g_{z_0}(z,\zeta,\lambda)}$ for a $\matn$-valued potential $v$.
In order to make use of \eqref{aless3} we define
\begin{align*}
u_1(z) &= \psi_{1, z_0}(z,\lambda)=e^{\lambda (z-z_0)^2} \mu_1(z,\lambda), \\
u_2(z) &=\overline \psi_{2, z_0}(z,-\lambda)=e^{- \bar \lambda (\bar z-\bar z_0)^2} \bar \mu_2(z,-\lambda), 
\end{align*}
for $z_0 \in D$, $\lambda \in C$, $|\lambda| > \rho$ ($\rho$ is mentioned in section 2), where we called for simplicity $\mu_1 = \mu_{1,z_0}$, $\mu_2 = \mu_{2,z_0}$ and $\mu_{1,z_0}$, $\mu_{2,z_0}$ are the solutions of \eqref{eq14} with $v$ replaced by $ v_1$, ${}^t v_2$, respectively. 

%and $1 \leq i \leq n$, the vector-valued function $\psi^i_{z_0}(z,\lambda)$ as the i-th column of the matrix-valued function $\psi_{z_0}(z,\lambda)$ defined above; the function $\mu^i_{z_0} (z,\lambda) = e^{-\lambda (z-z_0)^2} \psi^i_{z_0}(z,\lambda)$ satisfies Lemmata \ref{lem1}-\ref{lem4} (with modified constants). We have that 
%\begin{align*}
%(-\Delta + v_k)&\psi^i_{k, z_0}(z,\lambda)= 0 \textrm{ in } \C,\\
%&\psi^i_{k,z_0}(z,\lambda) \to e^{\lambda(z-z_0)^2}e_i \textrm{ as } z \to \infty
%\end{align*}
%for $k=1,2$ and $1 \leq i \leq n$ ($e_i$ is the i-th vector of the canonical basis of $\R^n$).
%
%For $1\leq i,j \leq n$, let
%\begin{align*}
%u_1(z) &=\overline \psi^i_{1, z_0}(z,-\lambda)=e^{- \bar \lambda (\bar z-\bar z_0)^2} \bar \mu^i_1(z,-\lambda), \\ 
%u_2(z) &= \psi^j_{2, z_0}(z,\lambda)=e^{\lambda (z-z_0)^2} \mu^j_2(z,\lambda), 
%\end{align*}
%where we called for simplicity $\mu^i_1 = \mu^i_{1,z_0}$ and $\mu^j_2 = \mu^j_{2,z_0}$.
Equation \eqref{aless3}, with the above-defined $u_1, u_2$, now reads
\begin{align} \label{aless4}
\int_{\partial D} \int_{\partial D} e^{-\bar \lambda (\bar z - \bar z_0)^2}\, {}^t \bar \mu_2(z,-\lambda) (\Phi_2 - \Phi_1)(z,\zeta) e^{\lambda (\zeta - z_0)^2} \mu_1(\zeta ,\lambda)|d\zeta|| d z |\\ \nonumber
= \int_D e_{\lambda,z_0} (z)\, {}^t \bar \mu_2(z,-\lambda) (v_2 - v_1)(z) \mu_1(z,\lambda) d \Ree z \, d \Imm z.
\end{align}
with $e_{\lambda,z_0}(z) = e^{\lambda(z-z_0)^2 - \bar \lambda (\bar z - \bar z_0)^2}$ and $(\Phi_2 - \Phi_1)(z,\zeta)$ is the Schwartz kernel of the operator $\Phi_2 -\Phi_1$.

The right side $I(\lambda)$ of \eqref{aless4} can be written as the sum of four integrals, namely
\begin{align*}
I_1(\lambda)&=\int_D e_{\lambda,z_0} (z)(v_2 - v_1)(z) d \Ree z \, d \Imm z, \\
I_2(\lambda)&=\int_D e_{\lambda,z_0}(z)\, {}^t \!(\bar \mu_2 -I) (v_2 - v_1)(z)( \mu_1 -I)  d \Ree z \, d \Imm z, \\
I_3(\lambda)&=\int_D e_{\lambda,z_0}(z)\, {}^t \!(\bar \mu_2 -I) (v_2 - v_1)(z) \, d \Ree z \, d \Imm z, \\
I_4(\lambda)&=\int_D e_{\lambda,z_0}(z)\, (v_2 - v_1)(z)( \mu_1 -I)  d \Ree z \, d \Imm z,
\end{align*}
for $z_0 \in D$.

The first term, $I_1$, can be estimated using Lemma \ref{lem2} as follows:
\begin{align} \label{est91}
&\left| \frac{2}{\pi}|\lambda| I_1 - (v_2(z_0) - v_1(z_0)) \right| \leq c_3( D,n) \frac{\log (3 |\lambda|)}{|\lambda|} \|v_2 - v_1\|_{C^2(\bar D)},
\end{align}
for $|\lambda| \geq 1$.
The other terms, $I_2, I_3, I_4$, satisfy, by Lemmata \ref{lem1} and \ref{lem4},
\begin{align} \label{est92}
&|I_2| \leq \left| \int_D e_{\lambda,z_0}(z)\, {}^t \!(\overline{g_{z_0,\lambda}}{}^t v_2) (v_2 - v_1)(z)(g_{z_0,\lambda}v_1)  d \Ree z \, d \Imm z \right|\\ \nonumber
&\qquad + O \left( \frac{\log(3|\lambda|)}{|\lambda|^2} \right)c_8(D,N,n), \\ \label{est93}
&|I_3| \leq  \left| \int_D e_{\lambda,z_0}(z)\, {}^t \!(\overline{g_{z_0,\lambda}} {}^t v_2) (v_2 - v_1)(z)  d \Ree z \, d \Imm z \right|\\ \nonumber
&\qquad + O \left( \frac{\log(3|\lambda|)}{|\lambda|^2} \right)c_9(D,N,n), \\ \label{est94}
&|I_4| \leq  \left| \int_D e_{\lambda,z_0}(z)\,(v_2 - v_1)(z)(g_{z_0,\lambda}v_1)  d \Ree z \, d \Imm z \right| \\ \nonumber
&\qquad + O \left( \frac{\log(3|\lambda|)}{|\lambda|^2} \right)c_{10}(D,N,n),
\end{align}
where $N$ is the costant in the statement of Theorem \ref{theo1} and $|\lambda|$ is sufficiently large, for example for $\lambda$ such that
\begin{align} \label{estlambda}
2n\frac{c_2(D)}{|\lambda|^{\frac 1 2 }} \le \frac 1 2 , \qquad |\lambda|\ge 1.
\end{align}

Lemmata \ref{lem5}, \ref{lem6}, applied to \eqref{est92}-\eqref{est94}, give us
\begin{align}
&|I_2| \leq c_{11}(D,N,n)\frac{\left( \log(3|\lambda|)\right)^2 }{|\lambda|^2}, \\
&|I_3| \leq c_{12}(D,N,n)\frac{\left(\log(3|\lambda|)\right)^2}{|\lambda|^{1+3/4}}, \\
&|I_4| \leq c_{13}(D,N,n)\frac{\left(\log(3|\lambda|)\right)^2}{|\lambda|^{1+3/4}}.
\end{align}

The left side $J(\lambda)$ of \eqref{aless4} can be estimated as follows:
\begin{align} \label{est95}
|\lambda||J(\lambda)| \leq c_{14}(D,n) e^{(2  L^2+1) |\lambda|} \|\Phi_2 - \Phi_1\|,
\end{align}
for $\lambda$ which satisfies \eqref{estlambda}, and $L = \max_{z \in \partial D, \; z_0 \in D} |z-z_0|$.

Putting together estimates \eqref{est91}-\eqref{est95} we obtain
\begin{align} \label{est96}
|v_2(z_0) - v_1(z_0)| \leq c_{15}(D,N,n) \frac{\left(\log(3 |\lambda|)\right)^2}{|\lambda|^{3/4}} + \frac{2}{\pi}c_{14}(D,n) e^{(2  L^2+1) |\lambda|} \|\Phi_2 - \Phi_1\|
\end{align}
for any $z_0 \in D$. We call $\varepsilon =  \|\Phi_2 - \Phi_1\|$ and impose $|\lambda| = \gamma \log (3+\varepsilon^{-1})$, where $0 < \gamma < (2L^2+1)^{-1}$ so that \eqref{est96} reads
\begin{align}
|v_2(z_0) - v_1(z_0)| &\leq c_{15}(D,N,n)(\gamma \log(3+\varepsilon^{-1}))^{-\frac 3 4} \left( \log(3 \gamma \log(3+\varepsilon^{-1}))\right)^2  \\ \nonumber
&\quad + \frac{2}{\pi}c_{14}(D,n) (3+ \varepsilon^{-1})^{(2  L^2+1)\gamma }\varepsilon,
\end{align}
for every $z_0 \in D$, with
\begin{equation} \label{esteps}
0 < \varepsilon \leq \varepsilon_1 (D, N, \gamma,n),
\end{equation}
where $\varepsilon_1$ is sufficiently small or, more precisely, where \eqref{esteps} implies that $|\lambda| =\gamma \log (3+ \varepsilon^{-1})$ satisfies \eqref{estlambda}.

As $(3+ \varepsilon^{-1})^{(2  L^2+1) \gamma} \varepsilon \to 0$ for $\varepsilon \to 0$ more rapidly then the other term, we obtain that
\begin{align} \label{estest}
\|v_2 -v_1\|_{L^{\infty}(D)} \leq c_{16}(D,N,\gamma,n) \frac{\left(\log (3\log (3+\|\Phi_2 - \Phi_1\|^{-1}))\right)^2}{\left(\log (3+ \|\Phi_2 - \Phi_1\|^{-1})\right)^{\frac 3 4}} 
\end{align}
for any $\varepsilon = \|\Phi_2 - \Phi_1\| \leq \varepsilon_1(D, N,\gamma,n)$.

Estimate \eqref{estest} for general $\varepsilon$ (with modified $c_{16}$) follows from \eqref{estest} for $\varepsilon \leq \varepsilon_1(D, N, \gamma,n)$ and the assumption that $\|v_j\|_{L^{\infty}(D)} \leq N, \; j = 1,2$. This completes the proof of Theorem \ref{theo1}. \qed

\section{Proofs of Lemmata \ref{lem5}, \ref{lem6}, \ref{lemtech}.} \label{seclem}

\begin{proof}[Proof of Lemma \ref{lem5}]
We decompose the operator $g_{z_0,\lambda}$, defined in \eqref{green}, as the product $\frac{1}{4}T_{z_0,\lambda}\bar T_{z_0,\lambda}$, where
\begin{align}
&T_{z_0,\lambda} u(z) = \frac{1}{\pi} \int_D \frac{e^{-\lambda(\zeta-z_0)^2+\bar \lambda(\bar \zeta - \bar z_0)^2}}{z-\zeta } u(\zeta)d \Ree\zeta \, d \Imm \zeta, \\ \label{deftbar}
&\bar T_{z_0,\lambda} u(z) = \frac{1}{\pi} \int_D \frac{e^{\lambda(\zeta-z_0)^2-\bar \lambda(\bar \zeta - \bar z_0)^2}}{\bar z -\bar \zeta} u(\zeta)  d \Ree\zeta \, d \Imm \zeta,
\end{align}
for $z_0,\lambda \in \C$. From the proof of \cite[Lemma 3.1]{NS} we have the estimate
\begin{align} \label{estbar}
\|\bar T_{z_0,\lambda} u \|_{C(\bar D)} \leq \frac{\eta_1(D)}{|\lambda|^{1/2}}\|u\|_{C(\bar D)}+\eta_2(D)\frac{\log(3|\lambda|)}{|\lambda|}\left\|\frac{\partial u}{\partial \bar z}\right\|_{C(\bar D)}, 
\end{align}
for $u \in C^1_{\bar z}(\bar D)$, $z_0 \in D$, $|\lambda| \geq 1$.
As the kernel of $T_{z_0,\lambda}$ and $\bar T_{z_0,\lambda}$ are conjugated each other we deduce immediately
\begin{align}
\|T_{z_0,\lambda} u \|_{C(\bar D)} \leq \frac{\eta_1(D)}{|\lambda|^{1/2}}\|u\|_{C(\bar D)}+\eta_2(D)\frac{\log(3|\lambda|)}{|\lambda|}\left\|\frac{\partial u}{\partial z}\right\|_{C(\bar D)}, \; |\lambda| \geq 1,
\end{align}
for $u \in C^1_{z}(\bar D)$.
Combining the two estimates we obtain
\begin{align*}
&\|g_{\lambda,z_0}u\|_{C(\bar D)}= \frac{1}{4}\|T_{z_0,\lambda} \bar T_{z_0,\lambda} u \|_{C(\bar D)} \\
&\qquad \leq \frac{1}{4}\left(\eta_1(D)\frac{\|\bar T_{z_0,\lambda}u\|_{C(\bar D)}}{|\lambda|^{1/2}}+\eta_2(D)\frac{\log(3|\lambda|)}{|\lambda|}\left\|\frac{\partial}{\partial z}\bar T_{z_0,\lambda}u\right\|_{C(\bar D)}\right)\\
&\qquad \leq \eta_3(D)\left(\frac{\|u\|_{C(\bar D)}}{|\lambda|}+\frac{\log(3|\lambda|)}{|\lambda|^{3/2}}\left\|\frac{\partial u}{\partial \bar z}\right\|_{C(\bar D)} +\frac{\log(3|\lambda|)}{|\lambda|}\|u\|_{C(\bar D)}\right) \\
&\qquad \leq \eta_4(D) \frac{\log(3|\lambda|)}{|\lambda|}\|u\|_{C^1_{\bar z}(\bar D)}, \qquad |\lambda| \geq 1,
\end{align*}
where we used the fact that $\| \frac{\partial}{\partial z}\bar T_{z_0,\lambda}u\|_{C(D)} = \|u\|_{C(D)}$.
\end{proof}

\begin{proof}[Proof of Lemma \ref{lem6}]
For $0<\varepsilon \leq 1$, $z_0 \in D$, let $B_{z_0,\varepsilon}=\{z \in \C : |z-z_0|\leq \varepsilon\}$. 
We write $W(u,v)(\lambda) = W^1(\lambda) + W^2(\lambda)$, where
\begin{align*}
&W^1(\lambda) =  \int_{D \cap B_{z_0,\varepsilon}}\! \! \! \! \! \! \! \! e^{\lambda (z-z_0)^2 -\bar \lambda (\bar z- \bar z_0)^2}  u(z) g_{z_0,\lambda}v(z) d \Ree z \, d \Imm z, \\
&W^2(\lambda) =  \int_{D \setminus B_{z_0,\varepsilon}}\! \! \! \! \! \! \! \! e^{\lambda (z-z_0)^2 -\bar \lambda (\bar z- \bar z_0)^2}  u(z) g_{z_0,\lambda}v(z) d \Ree z \, d \Imm z.
\end{align*}
The first term, $W^1$, can be estimated as follows:
\begin{align} \label{estranz}
|W^1(\lambda)| \leq \sigma_1(D,n) \|u\|_{C(\bar D)} \|v\|_{C^1_{\overline z}(\bar D)} \frac{\varepsilon^2  \log(3 |\lambda|)}{|\lambda|}, \qquad |\lambda| \geq 1,
\end{align}
where we used estimates \eqref{estsolmu} and \eqref{estcontg}.

For the second term, $W^2$, we proceed using integration by parts, in order to obtain 
\begin{align*}
W^2(\lambda) &= \frac{1}{4 i \bar \lambda }\int_{\partial (D \setminus B_{z_0,\varepsilon})} \! \! \! \! \! \! \! \! \! \! \! \! \! \! \! \! \! e^{\lambda (z-z_0)^2 -\bar \lambda (\bar z- \bar z_0)^2}\frac{u(z) g_{z_0,\lambda}v(z)}{\bar z - \bar z_0} dz\\
&\quad - \frac{1}{2 \bar \lambda}\int_{D \setminus B_{z_0,\varepsilon}} e^{\lambda (z-z_0)^2 -\bar \lambda (\bar z- \bar z_0)^2}\frac{\partial}{\partial \bar z}\left( \frac{u(z) g_{z_0,\lambda}v(z)}{\bar z-\bar z_0}\right) d \Ree z \, d \Imm z. 
\end{align*}
This imply
\begin{align} \label{ultima}
&|W^2(\lambda)| \leq \frac{1}{4 |\lambda|} \int_{\partial (D \setminus B_{z_0,\varepsilon})} \! \! \! \! \! \! \! \! \! \! \! \! \! \! \! \! \! \frac{\|u(z) g_{z_0,\lambda}v(z)\|_{C(\bar D)}}{|\bar z - \bar z_0|}|dz| \\   \nonumber
&\quad + \frac{1}{2|\lambda|}\left|  \int_{D \setminus B_{z_0,\varepsilon}}\! \! \! \! \! \! \! \! \! \! e^{\lambda (z-z_0)^2 -\bar \lambda (\bar z- \bar z_0)^2}\frac{\partial}{\partial \bar z}\left( \frac{u(z) g_{z_0,\lambda}v(z)}{\bar z-\bar z_0}\right) d \Ree z \, d \Imm z  \right|,
\end{align}
for $\lambda \neq 0$.
Again by estimates \eqref{estsolmu} and \eqref{estcontg} we obtain
\begin{align} \label{perdio}
&|W^2(\lambda)| \leq \sigma_2(D,n) \|u\|_{C^1_{\overline z}(\bar D)} \|v\|_{C^1_{\overline z}(\bar D)} \frac{\log(3 \varepsilon^{-1}) \log(3 |\lambda|) }{|\lambda|^{2}} \\ \nonumber
&\qquad \qquad+  \frac{1}{8|\lambda|} \left|\int_{D \setminus B_{z_0,\varepsilon}}u(z) \frac{\bar T_{z_0, \lambda}v(z)}{\bar z-\bar z_0} d\Ree z \, d \Imm z \right|, \qquad |\lambda| \geq 1,
\end{align}
where we used the fact that $\frac{\partial}{\partial \bar z}g_{z_0,\lambda}v(z) =\frac 1 4  e^{-\lambda (z-z_0)^2 +\bar \lambda (\bar z- \bar z_0)^2}\bar T_{z_0, \lambda}v(z)$, with $\bar T_{z_0,\lambda}$ defined in \eqref{deftbar}.

The last term in \eqref{perdio} can be estimated independently on $\varepsilon$ by
\begin{equation} \label{estopam}
\sigma_3(D,n)\|u\|_{C(\bar D)} \|v\|_{C^1_{\bar z}(\bar D)} \frac{\log(3|\lambda|)}{|\lambda|^{1+3/4}}.
\end{equation}
This is a consequence of \eqref{estbar} and of the estimate
\begin{equation} \label{estbar2}
|\bar T_{z_0,\lambda}u(z)| \leq \frac{\log (3 |\lambda|)(1+|z-z_0|)\tau_1(D)}{|\lambda| |z- z_0|^2}\|u\|_{C^1_{\bar z}(\bar D)}, \; |\lambda| \geq 1,
\end{equation}
for $u \in C^1_{\bar z}(\bar D)$, $z,z_0 \in D$ (a proof of \eqref{estbar2} can be found in the proof of \cite[Lemma 3.1]{NS}).

Indeed, for $0< \delta \leq \frac{1}{2}$ we have
\begin{align*}
&\left| \int_D u(z)\frac{\bar T_{z_0, \lambda} v(z)}{\bar z - \bar z_0} d \Ree z \, d \Imm z \right|  \\ \nonumber
&\quad \leq \int_{B_{z_0,\delta} \cap D} \! \! \! \! \! \! \! \! \! \! \! \! |u(z)| \frac{|\bar T_{z_0, \lambda} v(z)|}{|z - z_0|} d \Ree z \, d \Imm z + \int_{D \setminus B_{z_0,\delta}} \! \! \! \! \! \! \! \! \!  \! \! \! |u(z)| \frac{|\bar T_{z_0, \lambda} v(z)|}{|z - z_0|} d \Ree z \, d \Imm z \\
&\quad \leq  \|u\|_{C(\bar D)} \|v\|_{C^1_{\bar z}(\bar D)} \frac{\tau_2(D,n)}{|\lambda|^{1/2}} \int_{B_{z_0,\delta} \cap D} \frac{d \Ree z \, d \Imm z}{|z - z_0|} \\
&\qquad+ \|u\|_{C(\bar D)} \|v\|_{C^1_{\bar z}(\bar D)}\frac{\log(3|\lambda|)}{|\lambda|}\tau_3(D,n)\int_{D \setminus B_{z_0,\delta}}  \frac{d \Ree z \, d \Imm z}{|z - z_0|^3} \\
&\quad \leq  2 \pi \|u\|_{C(\bar D)} \|v \|_{C^1_{\bar z}(\bar D)}\tau_2(D,n) \frac{ \delta}{|\lambda|^{\frac 1 2}} + \|u\|_{C(\bar D)}\|v\|_{C^1_{\bar z}(\bar D)}\tau_4(D,n) \frac{\log (3 |\lambda|)}{|\lambda| \delta}, 
\end{align*}
for $|\lambda| \geq 1$. Putting $\delta = \frac 1 2 |\lambda|^{-1/4}$ in the last inequality gives \eqref{estopam}.

Finally, defining $\varepsilon=|\lambda|^{-1/2}$ in \eqref{perdio}, \eqref{estranz} and using \eqref{estopam}, we obtain the main estimate \eqref{estdoppiavu}, which thus finishes the proof of Lemma \ref{lem6}.
\end{proof}

\begin{proof}[Proof of Lemma \ref{lemtech}]
Take $u \in H^1(D,\matn)$ such that $(-\Delta + {}^t v) u = 0$ on $D$ and $u|_{\partial D} =0$. We want to prove that $u \equiv 0$ on $D$.

By our hypothesis, for any $f \in C^1(\partial D, \matn)$ there exists a unique $\tilde f \in H^1(D, \matn)$ such that $(-\Delta + v)\tilde f =0$ on $D$ and $\tilde f|_{\partial D}=f$. Thus we have, using Green's formula \eqref{greenform},
\begin{align*}
&\int_{\partial D} {}^t \! \left(\frac{\partial u}{\partial \nu} \right) f |dz| =  \int_{D} \left( {}^t \! \left( \Delta u \right) \tilde f - {}^t u \Delta \tilde f \right) d \Ree z \, d\Imm z \\
&\qquad = \int_D \left( {}^t \! \left({}^t v \, u \right) \tilde f - {}^t u \, v \, \tilde f \right) d \Ree z \, d\Imm z = 0
\end{align*}
which yields $\frac{\partial u}{\partial \nu}|_{\partial D} = 0$. Now consider the following straightforward generalization of Green's formula \eqref{greenform},
\begin{align}
\int_{\partial D}\left( {}^t \! \left( \frac{\partial f}{\partial \nu}\right) g - {}^t f \frac{\partial g}{\partial \nu} \right) |dz| = \int_D {}^t \! \left( (\Delta - {}^t v) f\right) g - {}^t f \left( (\Delta -v) g\right)  d \Ree z \, d\Imm z,
\end{align}
which holds (weakly) for any $f,g \in H^1(D, \matn)$. If we put $f = u$ we obtain
\begin{equation}
\int_D {}^t u \,(-\Delta+ v) g \, d \Ree z \, d\Imm z = 0
\end{equation}
for any $g \in H^1(D, \matn)$. By Fredholm alternative (see \cite[Sec. 6.2]{E}), for each $h \in L^2(D, \matn)$ there exists a unique $g \in H_0^1(D, \matn)=\{ g \in H^1(D, \matn) : g|_{\partial D}=0 \}$ such that $(-\Delta + v)g =h$: this yields $u \equiv 0$ on $D$. Thus Lemma \ref{lemtech} is proved.
\end{proof}

\section{An extension of Theorem \ref{theo1}} \label{secext}

As an extension of Theorem \ref{theo1} for the case when we do not assume that $v_j|_{\partial D} \equiv 0, \; \frac{\partial}{\partial \nu} v_j|_{\partial D} \equiv 0, \; j= 1,2$, we give the following result.
\begin{prop} \label{mainprop}
Let $D \subset \R^2$ be an open bounded domain with $C^2$ boundary, let $v_1 , v_2 \in C^2(\bar D,\matn)$ be two matrix-valued potentials which satisfy \eqref{direig}, with $\|v_j \|_{C^2(\bar D)} \le N$ for $j =1,2$, and $\Phi_1 , \Phi_2$ the corresponding Dirichlet-to-Neumann operators. Then, for any $0 < \alpha < \frac 1 5$, there exists a constant $C = C(D, N, n,\alpha)$ such that
\begin{equation} \label{estfaible}
\|v_2 - v_1\|_{L^{\infty}(D)} \leq C \left(\log(3 + \|\Phi_2 - \Phi_1\|_1^{-1} )\right)^{-\alpha},
\end{equation}
where $\| A\|_1$ is the norm for an operator $A : L^{\infty}(\partial D, \matn) \to L^{\infty}(\partial D,\matn)$, with kernel $A(x,y)$, defined as $\|A\|_1 = \sup_{x,y \in \partial D} |A(x,y)| (\log(3 + |x-y|^{-1}))^{-1}$ and $|A(x,y)| = \max_{1 \leq i,j \leq n} |A_{i,j}(x,y)|$.
\end{prop}

The only properties of $\| \: \|_1$ we will use are the following:
\begin{itemize}
\item[\it i)] $\|A \|_{L^{\infty}(\partial D) \to L^{\infty}(\partial D)} \leq const(D,n) \|A\|_1$;
\item[\it ii)] In a similar way as in formula (4.9) of \cite{N1} one can deduce
\begin{equation} \nonumber
\|v\|_{L^{\infty}(\partial D)} \leq const(n) \|\Phi_{v} - \Phi_0\|_1,
\end{equation}
for a matrix-valued potential $v$, $\Phi_v$ its associated Dirichlet-to-Neu\-mann operator and $\Phi_0$ the Dirichlet-to-Neumann operator of the $0$ potential.
\end{itemize}

We recall a lemma from \cite{NS}, which generalize Lemma \ref{lem2} to the case of potentials without boundary conditions. We define $(\partial D)_{\delta} = \{ z \in \C : dist(z, \partial D) < \delta \}$.

\begin{lem} \label{newlem2}
For $v \in C^2(\bar D)$ we have that
\begin{align} \label{newestm}
\left|v(z_0) - \frac{2}{\pi}|\lambda| h^{(0)}_{z_0}(\lambda)\right| &\leq \kappa_1(D,n)\delta^{-4} \frac{\log(3|\lambda|)}{|\lambda|}\|v\|_{C^2(\bar D)}\\ \nonumber
&\qquad + \kappa_2(D,n)\log(3+ \delta^{-1})\|v\|_{C( \partial D)},
\end{align}
for $z_0 \in D \setminus (\partial D)_{\delta}$, $0 < \delta < 1$, $\lambda \in \C$, $|\lambda| \geq 1$.
\end{lem}

The proof of Lemma \ref{newlem2} for the scalar case can be found in \cite{NS}: the generalization to the matrix-valued case is straightforward.

\begin{proof}[Proof of Proposition \ref{mainprop}]
Fix $0 < \alpha < \frac{1}{5}$ and $0 < \delta < 1$. We have the following chain of inequalities
\begin{align*}
&\|v_2 - v_1\|_{L^{\infty}(D)}\\
&\qquad = \max(\|v_2 - v_1\|_{L^{\infty}(D \cap (\partial D)_{\delta})}, \|v_2 - v_1\|_{L^{\infty}(D \setminus (\partial D)_{\delta})}) \\
&\qquad \leq C_1 \max \left( 2 N\delta + \|\Phi_2 - \Phi_1\|_1 , \frac{\log (3 \log(3+\| \Phi_2 -\Phi_1\|^{-1}))}{\delta^4 \log(3 + \|\Phi_2 - \Phi_1\|^{-1})} \right. \\ 
&\qquad \qquad \left. +\log(3+ \frac 1 \delta) \|\Phi_2 - \Phi_1\|_1 + \frac{\left(\log (3 \log(3+\| \Phi_2 -\Phi_1\|^{-1}))\right)^2}{(\log(3 + \|\Phi_2 - \Phi_1\|^{-1}))^{\frac 3 4}}\right) \\
&\qquad \leq C_2 \max \left( 2 N\delta + \|\Phi_2 - \Phi_1\|_1 , \frac{1}{\delta^4}\left( \log(3 + \|\Phi_2 - \Phi_1\|_1^{-1})\right)^{-5\alpha } \right. \\
&\qquad \qquad \left. +\log(3+ \frac 1 \delta) \|\Phi_2 - \Phi_1\|_1 +\frac{\left(\log (3 \log(3+\| \Phi_2 -\Phi_1\|_1^{-1}))\right)^2}{(\log(3 + \|\Phi_2 - \Phi_1\|_1^{-1}))^{\frac 3 4}} \right),
\end{align*}
where we followed the scheme of the proof of Theorem \ref{theo1} with the following modifications: we made use of Lemma \ref{newlem2} instead of Lemma \ref{lem2} and we also used {i)-ii)}; note that $C_1 = C_1(D, N,n)$ and $C_2 = C_2(D, N, n, \alpha)$.

Putting $\delta = \left(\log (3 + \| \Phi_2 - \Phi_1\|_1^{-1})\right)^{-\alpha}$ we obtain the desired inequality
\begin{align} \label{estfinal}
&\|v_2 - v_1\|_{L^{\infty}(D)} \leq C_3 \left(\log(3 + \|\Phi_2 - \Phi_1\|_1^{-1} )\right)^{-\alpha},
\end{align}
with $C_3 = C_3(D, N,n, \alpha)$, $ \|\Phi_2 - \Phi_1\|_1 = \varepsilon \leq \varepsilon_1(D,N,n,\alpha)$ with $\varepsilon_1$ sufficiently small or, more precisely when $\delta_1 = \left(\log (3 + \varepsilon_1^{-1})\right)^{-\alpha}$ satisfies: 
$$\delta_1 < 1,  \qquad \varepsilon_1 \leq 2N \delta_1,  \qquad \log(3+\frac{1}{\delta_1}) \varepsilon_1 \leq \delta_1.$$

Estimate \eqref{estfinal} for general $\varepsilon$ (with modified $C_3$) follows from \eqref{estfinal} for $\varepsilon \leq \varepsilon_1(D,N,n,\alpha)$ and the assumption that $\|v_j\|_{L^{\infty}(\bar D)} \leq N$ for $j=1,2$. This completes the proof of Proposition \ref{mainprop}.
\end{proof}

\end{document}